\documentclass[12pt,twoside]{amsart}
\usepackage{amssymb}
\usepackage{graphicx}
\usepackage[margin=1in]{geometry}

\newcommand{\beas}{\begin{eqnarray*}}
\newcommand{\eeas}{\end{eqnarray*}}
\newcommand{\bea}{\begin{eqnarray}}
\newcommand{\eea}{\end{eqnarray}}
\newcommand{\beq}{\begin{equation}}
\newcommand{\eeq}{\end{equation}}
\newcommand{\ben}{\begin{enumerate}}
\newcommand{\een}{\end{enumerate}}

\newtheorem{theorem}{Theorem}[section]
\newtheorem{lemma}[theorem]{Lemma}

\newtheorem{conjecture}[theorem]{Conjecture}

\theoremstyle{definition}
\newtheorem{remark}[theorem]{Remark}

\usepackage{latexsym}
\usepackage{color,hyperref}
\definecolor{darkblue}{rgb}{0,0,0.6}
\hypersetup{colorlinks,breaklinks,linkcolor=darkblue,urlcolor=darkblue,anchorcolor=darkblue,citecolor=darkblue}

\author[Richard P. Stanley]{Richard P. Stanley}
\address{Department of Mathematics, MIT, Cambridge, MA 02139}
\email{rstan@math.mit.edu}

\author[Fabrizio Zanello]{Fabrizio Zanello}
\address{Department of Mathematical Sciences, Michigan Tech, Houghton, MI 49931}
\email{zanello@mit.edu}

\title[A generalization of a 1998 unimodality conjecture of Reiner and Stanton]{A generalization of a 1998 unimodality conjecture\\of Reiner and Stanton}

\begin{document}

\begin{abstract} 
An interesting, and still wide open, conjecture of Reiner and Stanton predicts that certain ``strange'' symmetric differences of $q$-binomial coefficients are always nonnegative and unimodal. We extend their conjecture to a broader, and perhaps more natural, framework, by conjecturing that, for each $k\ge 5$, the polynomials
$$f(k,m,b)(q)=\binom{m}{k}_q-q^{\frac{k(m-b)}{2}+b-2k+2}\cdot\binom{b}{k-2}_q$$
are nonnegative and unimodal for all $m\gg_k 0$ and $b\le \frac{km-4k+4}{k-2}$ such that $kb\equiv km$ (mod 2), with the only exception of $b=\frac{km-4k+2}{k-2}$ when this is an integer.

Using the KOH theorem, we combinatorially show the case $k=5$. In fact, we completely characterize the nonnegativity and unimodality of $f(k,m,b)$ for $k\le 5$. (This also provides an isolated counterexample to Reiner-Stanton's conjecture when $k=3$.) Further, we prove that, for each $k$ and $m$, it suffices to show our conjecture for the largest $2k-6$ values of $b$.
\end{abstract}


\maketitle

\section{Introduction and statement of the conjecture} 

An intriguing conjecture of Reiner and Stanton, that has remained open for nearly 20 years, predicts the nonnegativity and unimodality of certain symmetric differences of $q$-binomial coefficients (see \cite{RS}, Conjecture 9 or \cite{Stanton}, Conjecture 7). The goal of this note is to frame their admittedly ``strange'' statement into a broader and more natural combinatorial setting, and show, by means of Zeilberger's KOH theorem, several special cases of our conjecture. 

For integers $m\ge k\ge 0$, define the \emph{$q$-binomial coefficient}
$$\binom {m}{k}_q=\frac{(1-q)(1-q^2)\cdots (1-q^{m})}{(1-q)(1-q^2)\cdots (1-q^k)\cdot (1-q)(1-q^2)\cdots (1-q^{{m-k}})}.$$
It is a well-known fact in combinatorics that $\binom {m}{k}_q$ is a unimodal, symmetric polynomial in $q$ of degree $k(m-k)$ with nonnegative integer coefficients (see e.g. \cite{OH,Pr1,St5,Sy}). 

Now let
$$f(k,m,b)(q)=\binom{m}{k}_q-q^{\frac{k(m-b)}{2}+b-2k+2}\cdot\binom{b}{k-2}_q,$$
for integers $k\ge 2$, $m\ge k$, and $k-2\le b\le \frac{km-4k+4}{k-2}$ for $k\ge 3$, where $k(m-b)$ is even. 

Note that $f(k,m,b)$ is also a symmetric polynomial, since both terms in the difference are symmetric about the same degree, $k(m-k)/2$.

\begin{conjecture}\label{conj}
Let $k$, $m$, and $b$ be as above. We have:
\begin{itemize}
\item[ (i)] If $k=2$, then $f(k,m,b)$ is nonnegative and unimodal if and only if $m$ is even.

\item[ (ii)] If $k=3$, then $f(k,m,b)$ is nonnegative and unimodal if and only if $b\ne 3m-10$ and: $b\neq 2$ if $m$ is even; $b\neq 1,5$ if $m\equiv 1$ (mod 4); and $b\neq 3$ if $m\equiv 3$ (mod 4).

\item[ (iii)] If $k=4$, then $f(k,m,b)$ is nonnegative and unimodal if and only if $b$ is even and $m\neq 5$.

\item[ (iv)] If $k\ge 5$, then $f(k,m,b)$ is nonnegative and unimodal for all $m\gg_k 0$ and all $b$, with the only exception of $b=\frac{km-4k+2}{k-2}$ when this is an integer (i.e., when $k-2$ divides $2m-6$).
\end{itemize}
\end{conjecture}

Notice that, when $b=\frac{km-4k+2}{k-2}$,  the $q$-binomial coefficient $\binom{b}{k-2}_q$ is shifted by exactly one degree, and therefore, for $m$ large enough, $\binom{m}{k}_q-q\cdot\binom{b}{k-2}_q=1+0q+q^2+\dots$ is never unimodal.

Hence, what our conjecture claims is essentially that, for any $k\ge 5$, $f(k,m,b)$ is unimodal ``as often as possible,'' provided that $m$ be sufficiently large. The small values of $m$, perhaps not unexpectedly, may still allow the shifted first difference of $\binom{b}{k-2}_q$ to grow faster than that of $\binom{m}{k}_q$. Computationally, however, these values do not appear to be too large relatively to $k$. (For instance, it seems safe to assume $m\ge 20$ for $k=5$; $m\ge 32$ for $k=6$; $m\ge 18$ for $k=7$ and $8$; $m\ge 20$ for $k=9$; $m\ge 24$ for $k=10$; etc..)

Also note that Reiner-Stanton's conjecture (\cite{RS}, Conjecture 9) essentially corresponds to the special case $m$ even, $b\ge m-4$, and $b\equiv m$ (mod 4) of our conjecture. Their Theorems 1 and 5 correspond to the cases $b=m$ even, and $b=m-2$ with $m-k$ even, respectively.

Our main results will be a proof of parts (i) to (iii) of  our conjecture, and  of part (iv) for $k= 5$. As a pithy application of the KOH theorem, we will also show that in order to prove the conjecture for any given $k$ and $m$, it suffices to do so for the \emph{$2k-6$ largest values} of $b$ (or for only  $k-3$ such values when $k$ is odd). In general, however, our conjecture remains open.

We finally note that the original conjecture of Reiner and Stanton has a counterexample when $k=3$, corresponding to their parameters $n=7$ and $r=0$ (for us, $m=6$ and $b=2$), which give the nonunimodal polynomial
$$\binom{6}{3}_q-q^4\binom{2}{1}_q=1+q+2q^2+3q^3+2q^4+2q^5+3q^6+2q^7+q^8+q^9.$$
Our results imply that, for $k\le 5$, this is the unique counterexample to their conjecture. 

\section{Proofs}

We begin with a crucial lemma, namely Zeilberger's KOH theorem \cite{Zei}. It provides an algebraic reformulation of O'Hara's combinatorial proof of the unimodality of $q$-binomial coefficients \cite{OH}, by decomposing $\binom{a+k}{k}_q$ into suitable finite sums of nonnegative, unimodal polynomials, all symmetric about degree $ak/2$.  

We fix  positive integers $a$ and $k$, and for any  partition $\lambda=(\lambda_1,\lambda_2,\dots)\vdash k$, define $Y_i= \sum_{1\le j\le i}\lambda_j$ for $i\geq 1$, and $Y_0=0$. Then:
\begin{lemma} [\cite{Zei}]\label{koh}
We have $\binom{a+k}{k}_q=\sum_{\lambda\vdash k}F_{\lambda}(q)$, where
$$F_{\lambda}(q)= q^{2\sum_{i\geq 1}\binom{\lambda_i}{2}} \prod_{j\geq 1} \binom{j(a+2)-Y_{j-1}-Y_{j+1}}{\lambda_j-\lambda_{j+1}}_q.$$
\end{lemma}

\begin{theorem}\label{444}
Conjecture \ref{conj} is true for $k\le 4$.
\end{theorem}

\begin{proof}
Set $$\binom{m}{k}_q=\sum_{i=0}^{k(m-k)}a_iq^i$$
and
$$q^{\frac{k(m-b)}{2}+b-2k+2}\cdot\binom{b}{k-2}_q=\sum_{i=0}^{k(m-k)}b_iq^i.$$
(Thus, in particular, $b_i=0$ for $i<\frac{k(m-b)}{2}+b-2k+2$.) For the sake of simplicity, we will identify, with some slight abuse of notation, the first difference of a symmetric polynomial with its truncation after the middle degree; hence $(1-q) \binom{m}{k}_q$ will denote the polynomial $1+\sum_{i=1}^{\left\lfloor k(m-k)/2\right\rfloor}(a_i-a_{i-1})q^i$, and similarly,
$$(1-q)q^{\frac{k(m-b)}{2}+b-2k+2}\cdot\binom{b}{k-2}_q=q^{\frac{k(m-b)}{2}+b-2k+2}+\sum_{i=\frac{k(m-b)}{2}+b-2k+3}^{\left\lfloor
  k(m-k)/2\right\rfloor}(b_i-b_{i-1})q^i.$$ 

It is clear that showing the theorem is tantamount to proving
$$(1-q)\binom{m}{k}_q\ge (1-q)q^{\frac{k(m-b)}{2}+b-2k+2}\cdot\binom{b}{k-2}_q,$$
where the partial order on polynomials is defined by setting $\sum
\alpha_iq^i \ge \sum \beta_iq^i$  whenever $\alpha_i \ge \beta_i$ for
all $i$.\\ 
\\
(i) The case $k=2$ is trivial. We have
 $$ f(2,m,b)(q)=\binom{m}{2}_q-q^{m-2}, $$
which is independent of $b$. Note that
$\binom{m}{2}_q=\sum_{i=0}^{2m-4}a_iq^i$ satisfies $a_i=\lfloor
(i+2)/2\rfloor$ for all $i\le m-2$. Therefore, $(1-q)\binom{m}{2}_q$
equals 0 in the odd degrees and 1 in the even degrees. Since the first
difference of $q^{m-2}$ (defined up until degree $m-2$) is clearly 1
in degree $m-2$ and 0 elsewhere, it immediately follows that
$f(2,m,b)$ is unimodal if and only if $m-2$ is even.\\ 
\\
(ii) Let $k=3$. We have:
\begin{equation}\label{3}
f(3,m,b)(q)=\binom{m}{3}_q-q^{\frac{3m-b-8}{2}}\cdot \binom{b}{1}_q,
\end{equation}
for $m\ge 3$, $1\le b\le 3m-8$, and $b\equiv m$ (mod 2). Note that
(\ref{3}) can be rewritten as 
$$\binom{m}{3}_q-\left(q^{\frac{3m-b-8}{2}}+q^{\frac{3m-b-6}{2}}+\dots+q^{\frac{3m+b-10}{2}}\right).$$ 

It follows that $f(3,m,b)$ is unimodal whenever $\binom{m}{3}_q$ strictly increases from degree $(3m-b-10)/2$ to $(3m-b-8)/2$. Note that this never happens when $(3m-b-8)/2=1$, i.e., $b=3m-10$.

It is now easy to see that the theorem is proven if we show that $\binom{m}{3}_q$ does not strictly increase (i.e., it is constant) from degree $j-1$ to $j$ precisely for the following values of $j$ in the range $2\le j\le (3m-9)/2$: 
$$j=(3m-10)/2 {\ }{\ }{\ }\text{if $m$ is even};$$
\begin{equation}\label{i}
j=(3m-9)/2 {\ }\text{and}{\ } (3m-13)/2 {\ }{\ }{\ }\text{if $m\equiv 1$ (mod 4)};
\end{equation}
$$j=(3m-11)/2 {\ }{\ }{\ }\text{if $m\equiv 3$ (mod 4)}.$$

The $q$-binomials $\binom{m}{3}_q$ are fairly well understood (we even
know explicit symmetric chain decompositions for the corresponding
Young lattice $L(3,m-3)$; see e.g. \cite{lind}), so there are
several ways to prove (\ref{i}). Given that we are going to employ the
KOH theorem extensively later on, it is illustrative to give a proof
using Theorem \ref{koh}. We first note that  $\binom{m}{3}_q$ can be
decomposed as  
\begin{equation}\label{3333}
\binom{m}{3}_q=q^6 \binom{m-4}{3}_q+q^2\binom{m-4}{1}_q\binom{2m-7}{1}_q+\binom{3m-8}{1}_q ,
\end{equation}
where the first summand on the right side corresponds to the partition $(3)$ of 3, the second to $(2,1)$, and the third to $(1,1,1)$.

We next iterate Theorem \ref{koh} a total of $c=\lfloor m/4 \rfloor$
times on the right side of (\ref{3333}), noting that
$\binom{m-4c}{3}_q=0$ unless $m\equiv 3$ (mod 4), in which case
$\binom{m-4c}{3}_q=\binom{3}{3}_q=1$. We obtain: 
\begin{equation}\label{c}
\binom{m}{3}_q= \epsilon q^{(3m-9)/2}+\sum_{i=0}^{c-1} \left(
q^{6i+2}\binom{m-4i-4}{1}_q \binom{2m-8i-7}{1}_q +
q^{6i}\binom{3m-12i-8}{1}_q\right),
\end{equation}
where $\epsilon=1$ for $m\equiv 3$ (mod 4) and $\epsilon=0$ otherwise.

Let us now compute the first difference $(1-q)\binom{m}{3}_q$ using
(\ref{c}). Since $(6i+2)+(2m-8i-7)$ and $6i+(3m-12i-8)$ are both
greater than $(3m-9)/2$ for $i\le c-1$, and $(1-q)\binom{m}{3}_q$ is
truncated after the middle degree $(3m-9)/2$, we have: 
$$(1-q)\binom{m}{3}_q$$$$=\epsilon q^{(3m-9)/2}+
\sum_{i=0}^{c-1}\left((1-q)q^{6i+2}\binom{m-4i-4}{1}_q \cdot
\frac{1-q^{2m-8i-7 }}{1-q}+(1-q)q^{6i}\cdot
\frac{1-q^{3m-12i-8}}{1-q}\right)$$ 
$$=\epsilon
q^{(3m-9)/2}+\sum_{i=0}^{c-1}
\left( q^{6i+2}(1+q+\dots+q^{m-4i-5})+q^{6i}\right)$$
 $$=\epsilon q^{(3m-9)/2}+\sum_{i=0}^{c-1}
 \left(q^{6i}+q^{6i+2}+q^{6i+3}+\dots+q^{m+2i-3}\right).$$  

It is now a simple exercise for the reader to check that the degrees
$j\le (3m-9)/2$ where the last displayed summation has coefficient
zero are exactly the values of $j$ indicated in (\ref{i}). This
completes the proof for $k=3$.\\ 
\\
(iii) Let $k=4$. We have:
$$f(4,m,b)(q)=\binom{m}{4}_q-q^{2m-b-6}\cdot \binom{b}{2}_q,$$
where $m\ge 4$ and $2\le b\le 2m-6.$ We want to show that
$$(1-q)\binom{m}{4}_q\ge (1-q)q^{2m-b-6}\cdot \binom{b}{2}_q$$
(where as usual the polynomials on both sides are set to be zero after degree $2m-8$).

One moment's though gives that
\begin{equation}\label{44}
(1-q)q^{2m-b-6}\cdot \binom{b}{2}_q=\sum_{i=0}^{\left\lfloor \frac{b-2}{2}\right\rfloor}q^{(2m-b-6)+2i}.
\end{equation}

It is easy to check the result directly for $m\le 5$ (in particular, notice that for $m=5$ and $b=4$, $f(4,5,4)(q)=-q^2$ has a negative coefficient). Thus, assume $m\ge 6$. We want to determine when $\binom{m}{4}_q$ strictly increases from degree $(2m-b-6)+2i-1$ to $(2m-b-6)+2i$, for all $0\le i\le \left\lfloor \frac{b-2}{2}\right\rfloor$.  

The growth of $\binom{m}{4}_q$ can be studied in a few different ways (for instance, using our own \cite{SZ}, Lemma 2.1; or again via the KOH theorem; or, perhaps most instructively, using a symmetric chain decomposition for $L(4,m-4)$ \cite{west}). In particular, it can be seen that $\binom{m}{4}_q$  strictly increases from degree $j-1$ to $j$ for all even values of $j\le 2m-8$, and that it is always constant from degree $2m-10$ to $2m-9$. By (\ref{44}), this easily gives that for $m\ge 6$, $f(4,m,b)$ is nonnegative and unimodal if and only if $b$ is even, as desired.
\end{proof}

Our next result is an especially elegant application of the KOH theorem.

\begin{theorem}\label{2k-6}
Let $k\ge 4$, and $m$ and $b$ be as in Conjecture \ref{conj}. Assume that $f(k,m,b)$ is nonnegative and unimodal. Then $f(k,m,b-(2k-6))$ is also nonnegative and unimodal.
\end{theorem}

\begin{proof}
We want to show that if
$$(1-q)\binom{m}{k}_q\ge (1-q)q^{\frac{k(m-b)}{2}+b-2k+2}\cdot\binom{b}{k-2}_q,$$
then $(1-q)\binom{m}{k}_q$ also dominates
$$ (1-q)q^{\frac{k(m-(b-2k+6))}{2}+(b-2k+6)-2k+2}\cdot\binom{b-2k+6}{k-2}_q=(1-q)q^{\frac{k(m-b)}{2}+b-2k+2}\cdot q^{2\binom{k-2}{2}}\cdot\binom{b-2k+6}{k-2}_q.$$
To this end, it suffices to show
$$(1-q)q^{\frac{k(m-b)}{2}+b-2k+2}\cdot\binom{b}{k-2}_q \ge (1-q)q^{\frac{k(m-b)}{2}+b-2k+2}\cdot q^{2\binom{k-2}{2}}\cdot\binom{b-2k+6}{k-2}_q,$$
or equivalently,
\begin{equation}\label{delta}
(1-q)\binom{b}{k-2}_q \ge (1-q) q^{2\binom{k-2}{2}}\cdot\binom{b-2k+6}{k-2}_q.
\end{equation}

Consider the KOH decomposition of $\binom{b}{k-2}_q=\binom{(b-k+2)+(k-2)}{k-2}_q$, as in Theorem \ref{koh}. The term corresponding to the partition $(k-2)$ of $k-2$ is given by
$$q^{2\binom{k-2}{2}}\cdot \binom{(b-k+4)-(k-2)}{k-2}_q=q^{2\binom{k-2}{2}}\cdot \binom{b-2k+6}{k-2}_q.$$

Thus, $\binom{b}{k-2}_q$ decomposes as:
$$\binom{b}{k-2}_q=q^{2\binom{k-2}{2}}\cdot \binom{b-2k+6}{k-2}_q+\sum_{\lambda \neq (k-2)}F_{\lambda}(q),$$
where the sum on the right side is indexed over all partitions $\lambda \neq (k-2)$ of $k-2$. The crucial observation is that all the  $F_{\lambda}(q)$ are also unimodal and symmetric about the same degree, $(k-2)(b-k+2)$, and therefore their first differences are nonnegative.

We conclude that
$$(1-q)\binom{b}{k-2}_q=(1-q)q^{2\binom{k-2}{2}}\cdot \binom{b-2k+6}{k-2}_q+(1-q)\sum_{\lambda \neq (k-2)}F_{\lambda}(q)$$$$\geq (1-q) q^{2\binom{k-2}{2}}\cdot\binom{b-2k+6}{k-2}_q,$$
which is precisely (\ref{delta}).
\end{proof}

We next show our conjecture for $k=5$. Note that since our argument will explicitly assume $m\ge 20$ (and there exist only finitely many polynomials $f(5,m,b)$ when $5\le m\le 19$, all of which can easily be computed), this will completely characterize the nonnegativity and unimodality of $f(k,m,b)$ also for $k=5$.

\begin{theorem}
Conjecture \ref{conj} is true for $k=5$.
\end{theorem}

\begin{proof}
We will show that, for all $m\ge 20$ and $3\le b\le (5m-16)/3$ such that $b-m$ is even, 
$$f(5,m,b)(q)=\binom{m}{5}_q-q^{\frac{5m-3b-16}{2}}\cdot \binom{b}{3}_q$$
is nonnegative and unimodal, where for $m\equiv 0$ (mod 3) we further assume that $b\neq (5m-18)/3$, i.e., $b\le (5m-21)/3$.

Note that by Theorem \ref{2k-6}, since $2k-6=4$, for each $m$ it is
enough to show the theorem for the two largest values of $b$
satisfying the above conditions. More explicitly, standard
computations give that, if we write $m=6n+j\ge 20$ according to its
residue class modulo 6, we can reduce the problem to proving the nonnegativity and 
unimodality of the following twelve symmetric $q$-binomial
differences: 

\begin{equation}\label{n1}
\binom{6n}{5}_q-q^4 \binom{10n-8}{3}_q
\end{equation}

\begin{equation}\label{n2}
\binom{6n}{5}_q-q^7 \binom{10n-10}{3}_q
\end{equation}

\begin{equation}\label{n3}
\binom{6n+1}{5}_q-q^2 \binom{10n-5}{3}_q
\end{equation}

\begin{equation}\label{n4}
\binom{6n+1}{5}_q-q^5 \binom{10n-7}{3}_q
\end{equation}

\begin{equation}\label{n5}
\binom{6n+2}{5}_q- \binom{10n-2}{3}_q
\end{equation}

\begin{equation}\label{n6}
\binom{6n+2}{5}_q-q^3 \binom{10n-4}{3}_q
\end{equation}

\begin{equation}\label{n7}
\binom{6n+3}{5}_q-q^4 \binom{10n-3}{3}_q
\end{equation}

\begin{equation}\label{n8}
\binom{6n+3}{5}_q-q^7 \binom{10n-5}{3}_q
\end{equation}

\begin{equation}\label{n9}
\binom{6n+4}{5}_q-q^2 \binom{10n}{3}_q
\end{equation}

\begin{equation}\label{n10}
\binom{6n+4}{5}_q-q^5 \binom{10n-2}{3}_q
\end{equation}

\begin{equation}\label{n11}
\binom{6n+5}{5}_q- \binom{10n+3}{3}_q
\end{equation}

\begin{equation}\label{n12}
\binom{6n+5}{5}_q-q^3 \binom{10n+1}{3}_q.
\end{equation}

We present below the proof of (\ref{n1}), using the KOH theorem. We will show that for any $n\ge 4$ (in fact, the result is true for any $n$),
\begin{equation}\label{casen1}
(1-q)\binom{6n}{5}_q\ge (1-q)q^4 \binom{10n-8}{3}_q,
\end{equation}
where as usual both sides of (\ref{casen1}) are defined up until degree $\lfloor (5(6n-5)/2)\rfloor =15n-13$. 

By Theorem \ref{koh}, we can see that $\binom{6n}{5}_q=\binom{(6n-5)+5}{5}_q$ decomposes as:
$$\binom{6n}{5}_q=q^{20}\binom{6n-8}{5}_q+q^{12}\binom{6n-8}{3}_q\binom{12n-15}{1}_q$$
\begin{equation}\label{7}
+q^8\binom{6n-8}{1}_q\binom{12n-14}{2}_q+q^6\binom{6n-7}{2}_q\binom{18n-18}{1}_q
\end{equation}
$$+q^4\binom{12n-13}{1}_q\binom{18n-18}{1}_q+q^2\binom{6n-6}{1}_q\binom{24n-21}{1}_q+\binom{30n-24}{1}_q,$$
where the seven terms on the right side are all unimodal and symmetric about degree $(30n-25)/2$, and correspond to the following partitions of 5, respectively: $(5)$, $(4,1)$, $(3,2)$, $(3,1,1)$, $(2,2,1)$, $(2,1,1,1)$, and $(1,1,1,1,1)$.

Again by Theorem \ref{koh}, $q^4 \binom{10n-8}{3}_q=q^4 \binom{(10n-11)+3}{3}_q$ decomposes as
\begin{equation}\label{3koh3}
q^4 \binom{10n-8}{3}_q=q^{10}\binom{10n-12}{3}+q^6\binom{10n-12}{1}\binom{20n-23}{1}+q^4\binom{30n-32}{1},
\end{equation}
where the terms on the right side are contributed by, respectively, the partitions $(3)$, $(2,1)$, and $(1,1,1)$ of 3.

By iterating Theorem \ref{koh} a total of $c=\lfloor 5n/2\rfloor -2$ times on the right side of (\ref{3koh3}), similarly to what we did in the proof of Theorem \ref{444} for $k=3$, we eventually obtain:
$$q^4 \binom{10n-8}{3}_q=\sum_{i=0}^{c-1} q^{6i+6}\binom{10n-4i-12}{1}_q\binom{20n-8i-23}{1}_q + q^{6i+4}\binom{30n-12i-32}{1}_q.$$

Likewise, note that both $$(6i+6)+(20n-8i-23)=20n-2i-17$$ and $$(6i+4)+(30n-12i-32)=30n-6i-28$$ are greater than $15n-13$ for $i\le c-1$. Thus,
$$(1-q)q^4\binom{10n-8}{3}_q$$$$=\sum_{i=0}^{c-1}(1-q) q^{6i+6}\left(1+q+\dots+q^{10n-4i-13}\right)\cdot \frac{1-q^{20n-8i-23}}{1-q} +\sum_{i=0}^{c-1}(1-q) q^{6i+4}\cdot \frac{1-q^{30n-12i-32}}{1-q}$$$$=\sum_{i=0}^{c-1}q^{6i+6}\left(1+q+\dots+q^{10n-4i-13}\right)+q^{6i+4}$$
\begin{equation}\label{3delta}
=\sum_{i=0}^{c-1}q^{6i+4}+q^{6i+6}+q^{6i+7}+\dots+q^{10n+2i-7}.
\end{equation}

Now denote by $\sum_{i=0}^{15n-13}d_iq^i$ the summation in (\ref{3delta}). Our goal is to show that
$$(1-q)\binom{6n}{5}_q \ge \sum_{i=0}^{15n-13}d_iq^i.$$
We claim that the first difference of the contribution given by the partition $(3,2)$ to the KOH decomposition (\ref{7}) of $\binom{6n}{5}_q$, namely
$$q^8\binom{6n-8}{1}_q\binom{12n-14}{2}_q,$$
will suffice to dominate $\sum_{i=0}^{15n-13}d_iq^i$, in every degree $i\ge 8$.

In other words, if we let
$$(1-q)q^8\binom{6n-8}{1}_q\binom{12n-14}{2}_q=\sum_{i=0}^{15n-13}c_iq^i,$$
then we want to show that
$$c_i\ge d_i$$
for all $i=8,\dots,15n-13.$ This will prove the theorem, since it is
easy to see directly that $(1-q)\binom{6n}{5}_q$ dominates
(\ref{3delta}) in each degree $i\le 7$ (or, to overkill, one can
employ the first difference of the term
$q^4\binom{12n-13}{1}_q\binom{18n-18}{1}_q$ in those degrees). 

We first determine the $c_i$. Letting $\binom{12n-14}{2}_q=\sum_{i=0}^{24n-32}e_iq^i$, we have:
$$\sum_{i=0}^{15n-13}c_iq^i=(1-q)q^8\binom{6n-8}{1}_q\binom{12n-14}{2}_q$$$$=q^8(1-q^{6n-8})\sum_{i=0}^{24n-32}e_iq^i=(q^8-q^{6n})\sum_{i=0}^{24n-32}e_iq^i$$
(where the last expression is also considered to be zero after degree $15n-13$).

Thus, $c_i$ equals $e_{i-8}$ for $i\le 6n-1$, and it equals $e_{i-8}-e_{i-6n}$ for $i\ge 6n$. Standard computations therefore give us that:
$$c_i=\left\lfloor \frac{i-8+2}{2}\right \rfloor =\left\lfloor \frac{i}{2}\right \rfloor -3, {\ }{\ }\text{for}{\ }8\le i\le 6n-1;  $$
$$c_i=\left\lfloor \frac{i-8+2}{2}\right \rfloor -\left\lfloor \frac{i-6n+2}{2}\right \rfloor =3n-4, {\ }{\ }\text{for}{\ }6n\le i\le 12n-8;$$
$$c_i= \left\lfloor \frac{(24n-32)-(i-8)+2}{2}\right \rfloor-\left\lfloor \frac{i-6n+2}{2}\right \rfloor=15n-12-i, {\ }{\ }\text{for}{\ }12n-7\le i\le 15n-13.$$

Moving on to the coefficients of (\ref{3delta}), an elementary (but careful) computation explicitly determines the $d_i$, yielding:

$$d_i=\left\lfloor \frac{i+2}{6}\right \rfloor, {\ }{\ }\text{for}{\ }i\le 10n-7,{\ }i\not \equiv 5{\ }(\text{mod}{\ } 6);$$
\begin{equation}\label{di}
d_i=\left\lfloor \frac{i+2}{6}\right \rfloor-1, {\ }{\ }\text{for}{\ }i\le 10n-7,{\ }i \equiv 5{\ }(\text{mod}{\ } 6);
\end{equation}
$$d_i=\left\lceil \frac{15n-13-i}{3}\right \rceil, {\ }{\ }\text{for}{\ }10n-6\le i\le 15n-13.$$

It is now a trivial exercise to verify that $c_i\ge d_i$ for all $i=8,\dots,15n-13,$ as we wanted to show.  This proves (\ref{n1}). 

We leave most of the proof of the other eleven cases as an exercise to the interested reader, since except for three of them, the idea is entirely the same, though obviously some of the computations will differ. (Notice that for the two differences that present no shift, (\ref{n5}) and (\ref{n11}), the result is already known as a special case of \cite{Za2}, Theorem 2.4.) 

The three cases that present one substantial difference in the proof are: (\ref{n4}) when $n$ is odd, (\ref{n8}) when $n$ is even, and (\ref{n12}) when $n$ is odd.  Here, the same approach as above, using the term contributed by $(3,2)$ to the KOH decomposition of the corresponding $q$-binomial $\binom{6n+j}{5}_q$, suffices to prove nonnegativity and unimodality all the way up to the middle degree, \emph{except} for the middle degree itself. In particular, in all three cases, the corresponding first difference equals 0, while we would like at least 1. This issue can be solved by also employing the KOH contribution of $(4,1)$. We will outline this explicitly for (\ref{n4}), the other two cases being entirely similar.

Let $n=2t+1$ be odd, with $t\ge 2$. We can see that
$$(1-q)q^5 \binom{10n-7}{3}_q=(1-q)q^5\binom{20t+3}{3}_q$$
equals 1 in the largest degree, $30t+5$ (for instance, this follows from the proof of Theorem \ref{444}, since $20t+3\equiv 3$ (mod 4)). However, when we consider the term contributed by $(3,2)$ to the KOH decomposition of
$$\binom{6n+1}{5}_q=\binom{12t+7}{5}_q,$$
similarly to what we did in the proof of (\ref{n1}), its first difference in degree $30t+5$ turns out to be zero. (In all previous degrees, the desired inequality on the first differences holds.) Thus, we claim that the KOH contribution of the partition $(4,1)$, namely
$$q^{12}\binom{12t-1}{3}_q\binom{24t-1}{1}_q,$$
has a first difference of at least 1 in degree $30t+5$, which will complete the proof.

Setting
$$\binom{12t-1}{3}_q=\sum_{i=0}^{36t-12}\alpha_iq^i$$
and
$$(1-q)q^{12}\binom{12t-1}{3}_q\binom{24t-1}{1}_q=\sum_{i=0}^{30t+5}\delta_iq^i,$$
we obtain:
$$\sum_{i=0}^{30t+5}\delta_iq^i=(1-q)q^{12}\cdot \frac{1-q^{24t-1}}{1-q} \sum_{i=0}^{36t-12}\alpha_iq^i$$$$=\left(q^{12}-q^{24t+11}\right) \sum_{i=0}^{36t-12}\alpha_iq^i$$
(where, as usual, the last expression is also set to be zero after degree $30t+5$). We want to show that $\delta_{30t+5}>0$.

Notice that, by symmetry, $\alpha_i=\alpha_{36t-12-i}$ for all $i$. Therefore,
$$\delta_{30t+5}=\alpha_{30t-7}-\alpha_{6t-6}$$$$=\alpha_{(36t-12)-(30t-7)}-\alpha_{6t-6}=\alpha_{6t-5}-\alpha_{6t-6}.$$

Since $\binom{12t-1}{3}_q$ strictly increases from degree $6t-6$ to $6t-5$ (see for instance the proof of Theorem \ref{444} for $k=3$), the proof is complete.
\end{proof}

\begin{remark}
It seems likely that, with considerably more effort (and very tedious computations), an approach involving the KOH theorem might also work for the case $k=6$ of Conjecture \ref{conj}. For larger values of $k$, a new idea will most likely be necessary. 
\end{remark}

\begin{remark}
While in this paper we focused on what we believe to be a proper level
of generality for symmetric differences of $q$-binomial coefficients
whose denominators differ by 2, as in Reiner-Stanton's original
conjecture, in general  it seems natural to continue to expect
\emph{most} symmetric differences of $q$-binomials to be nonnegative and unimodal.
For the case where the two $q$-binomial coefficients have the same
degree (i.e., no shift by powers of $q$ is required), see the second
author's \cite{Za2}, where it is conjectured that \emph{all} such differences are nonnegative and unimodal. That conjecture is also open in general.  

When it comes to problems of this nature, a significant source of
difficulty --- and of interest --- is that we are still far from
having a complete understanding of the growth of the coefficients of
$\binom{m}{k}_q$, a basic question in this area of combinatorics. For
recent progress in this direction, we refer the reader to some of our
own work or that by Pak and Panova: \cite{PP,PP2,SZ2,Za}. 
\end{remark}

\section{Acknowledgements} This paper was written during a visiting
professorship of the second author in Fall 2017, for which he warmly
thanks the first author and the MIT Math Department. We wish to thank Fran\c{c}ois Bergeron, William Keith, and Greta Panova for various comments. The second author was partially supported by a Simons Foundation grant (\#274577). 


\end{document}